\documentclass[11pt]{article}
\usepackage{geometry}                
\usepackage{graphicx}
\usepackage{amssymb}
\usepackage{epstopdf}
\usepackage{amsthm}
\newtheorem{thm}{Theorem}[section]

\newtheorem{lemma}[thm]{Lemma}
\newtheorem{corr}[thm]{Corollary}
\newtheorem{fact}[thm]{Fact}
\newtheorem{defn}[thm]{Definition}

\usepackage[colorlinks=true, pdfstartview=FitV, linkcolor=blue,
            citecolor=blue, urlcolor=blue]{hyperref}

\title{Overlap Cycles for Permutations:  \\ Necessary and Sufficient Conditions}
\author{Victoria Horan\thanks{\texttt{victoria.horan.1@us.af.mil}} \\ Air Force Research Laboratory \\ Rome, NY  13441 USA}

\begin{document}

\maketitle

\begin{abstract}
	Universal cycles are generalizations of de Bruijn cycles and Gray codes that were introduced originally by Chung, Diaconis, and Graham in 1992. They have been developed by many authors since, for various combinatorial objects such as strings, subsets, permutations, partitions, vector spaces, and designs.  One generalization of universal cycles, which require almost complete overlap of consecutive words, is $s$-overlap cycles, which relax such a constraint.  In this paper we study permutations and some closely related class of strings, namely juggling sequences and functions.  We prove the existence of $s$-overlap cycles for these objects, as they do not always lend themselves to the universal cycle structure.
\end{abstract}

\section{Introduction}

\let\thefootnote\relax\footnote{Approved for public release; distribution unlimited:  88ABW-2013-4003, 9 Sep 13}

Listing structures for combinatorial objects are quickly becoming useful in more and more interesting applications.  Gray codes, first defined in 1947 by Frank Gray \cite{Gray}, are used in many different surprising places from position encoders \cite{Rotary} to genetic algorithms \cite{GA}.  More recently universal cycles are being used in areas such as rank modulation for multilevel flash memories \cite{Flash}, however overlap cycles are still being explored and have the potential to be useful in many applications.

An \textbf{$s$-overlap cycle}, or \textbf{$s$-ocycle}, is an ordering of a set of objects $\mathcal{C}$, each represented as a string of length $n$.  The ordering requires that object $b=b_0b_1 \ldots b_{n-1}$ follow object $a = a_0a_1 \ldots a_{n-1}$ only if $a_{n-s}a_{n-s+1} \ldots a_{n-1} = b_0b_1 \ldots b_{s-1}$.  Ocycles were introduced by Godbole, Knisley, and Norwood in 2010 \cite{Godbole}.  \textbf{Universal cycles}, or \textbf{ucycles} are $(n-1)$-ocycles and were originally introduced in 1992 by Chung, Diaconis, and Graham \cite{UC}.

To find $s$-ocycles and ucycles on a set of strings, most proofs employ the same method.  For a given string $X=x_1x_2 \ldots x_{n}$, let $X^{s-}=x_1x_2 \ldots x_s$ denote the $s$-prefix of $X$ and $X^{s+}=x_{n-s}x_{n-s+1} \ldots x_n$ denote the $s$-suffix of $X$.  The first step in the proof is to construct the \textbf{transition digraph} for the set of strings as follows.  Vertices represent $s$-prefixes and $s$-suffixes of strings (the overlaps), while each edge represents a string, traveling from its $s$-prefix to its $s$-suffix.  Note that the transition digraph is a directed multigraph in which an Euler tour (a closed walk in which every edge is traversed exactly once) corresponds bijectively to an $s$-ocycle.  To prove the existence of an Euler tour, we use the following well-known theorem.

\begin{thm}\label{euler}
	\emph{(\cite{West}, p. 60)}  A directed graph $G$ is eulerian if and only if it is both balanced and weakly connected.
\end{thm}

In this paper, we consider using the ocycle listing structure for permutations, as well as functions and juggling sequences.  We represent permutations of an $n$-set $\{0,1, \ldots , n-1\}$ as a string $\Pi = \pi_0 \pi_1 \ldots \pi_{n-1}$, where the functional representation is used, i.e. $\pi(0)=\pi_0$.  Closely related to permutations, juggling sequences have been an active research area since the 1980's \cite{BEGW}.  These sequences are used to determine what patterns a fixed set of balls can be juggled in, where only one ball may be thrown at a time.

\begin{defn}
    \emph{\cite{ChungGraham}}  A \textbf{juggling sequence} is a string $T=t_0t_1 \ldots t_{n-1}$ where each $t_i \geq 0$ such that $$\left| \{ i+t_i \pmod n \mid 0 \leq i \leq n-1\} \right| = n.$$
\end{defn}

This sequence illustrates that at time $i$, we should throw a ball high enough that it is in the air for $t_i$ beats, or to \textbf{height} $t_i$.  The number of \textbf{balls} used in a given juggling sequence $T=t_0t_1 \ldots t_{n-1}$ is given by $$b = \frac{1}{n} \sum_{i=0}^{n-1}t_i.$$  The \textbf{period} of a juggling sequence $t_0t_1 \ldots t_{n-1}$ is $n$, the length of the string.  

An alternative definition considers the corresponding \textbf{permutation sequence} for a juggling sequence.  Given a string $T=t_0t_1 \ldots t_{n-1}$, the permutation sequence is the string $\Pi_T = \pi_0 \pi_1 \ldots \pi_{n-1}$ where $\pi_i = t_i + i \pmod n$.  Then $\Pi_T$ is a permutation of the $n$-set $\{0, 1, \ldots , n-1\}$ if and only if $T$ is a valid juggling sequence.  From this definition it is clear that there is a very close relationship between juggling sequences and permutations, which is further illustrated by their similar results explored in this paper.  The permutation sequence clarifies the point that a juggler cannot catch or throw two objects simultaneously.

For example, when $n=3$ and $b=2$ the pattern $015$ is valid.  We can see this by drawing the \textbf{juggling diagram}, shown in Figure \ref{JD015}.  However, the sequence $105$ is not valid, as shown in Figure \ref{JD105}.  In this second example, note that on the second beat we are required two catch two balls simultaneously - an operation that is not allowed.

\begin{figure}
	\begin{center}
	\setlength{\unitlength}{10mm}
	\begin{picture}(8,5)
		\put(0,1){\circle{.2}}
        \put(0,0){0}
        \put(1,1){\circle{.2}}
        \put(1,0){1}
        \put(2,1){\circle{.2}}
        \put(2,0){5}
        \put(3,1){\circle{.2}}
        \put(3,0){0}
        \put(4,1){\circle{.2}}
        \put(4,0){1}
        \put(5,1){\circle{.2}}
        \put(5,0){5}
        \put(6,1){\circle{.2}}
        \put(6,0){0}
        \put(7,1){\circle{.2}}
        \put(7,0){1}
        \put(8,1){\circle{.2}}
        \put(8,0){5}
		\qbezier(1,1.1)(1.5,2)(2,1.1)
		\qbezier(2,1.1)(4.5,5.5)(7,1.1)
		\qbezier(0,3)(0.5,3)(1,1.1)
        \qbezier(0,2)(1.5,5)(4,1.1)
		\qbezier(4,1.1)(4.5,2)(5,1.1)
		\qbezier(7,1.1)(7.5,2)(8,1.1)
		\qbezier(5,1.1)(7.5,3.6)(8,3.2)
	\end{picture}
	\end{center}
	\caption{Juggling Diagram for Sequence $015$}\label{JD015}
\end{figure}
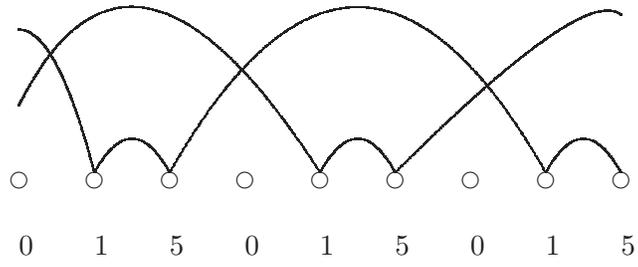

\begin{figure}
	\begin{center}
	\setlength{\unitlength}{10mm}
	\begin{picture}(8,5)
		\put(0,1){\circle{.2}}
        \put(0,0){1}
        \put(1,1){\circle{.2}}
        \put(1,0){0}
        \put(2,1){\circle{.2}}
        \put(2,0){5}
        \put(3,1){\circle{.2}}
        \put(3,0){1}
        \put(4,1){\circle{.2}}
        \put(4,0){0}
        \put(5,1){\circle{.2}}
        \put(5,0){5}
        \put(6,1){\circle{.2}}
        \put(6,0){1}
        \put(7,1){\circle{.2}}
        \put(7,0){0}
        \put(8,1){\circle{.2}}
        \put(8,0){5}
		\qbezier(0,1.1)(0.5,2)(1,1.1)
		\qbezier(2,1.1)(4.5,5.5)(7,1.1)
		\qbezier(0,3)(0.5,3)(1,1.1)
        \qbezier(0,2)(1.5,5)(4,1.1)
		\qbezier(3,1.1)(3.5,2)(4,1.1)
		\qbezier(6,1.1)(6.5,2)(7,1.1)
		\qbezier(5,1.1)(7.5,3.6)(8,3.2)
	\end{picture}
	\end{center}
	\caption{Juggling Diagram for Sequence $105$}\label{JD105}
\end{figure}
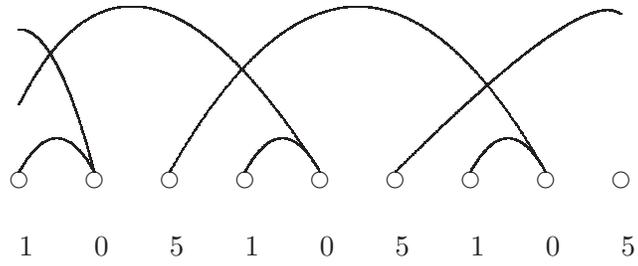

In \cite{ChungGraham}, Chung and Graham show that it is not always possible to find a single universal cycle that contains all juggling sequences of period $n$ and at most $b$ balls.  However they do prove that we can use several disjoint universal cycles to cover each sequence exactly once.  This result is closely related to the long studied problem of universal cycles for permutations.  Since ucycles for permutations are not possible using the standard representation, a modification of the representation can often provide an effective solution \cite{Isaak, PermsUC}.

An alternative that has only recently been explored is to modify the listing structure rather than the object representation.  By utilizing $s$-ocycles rather than ucycles, partial results for permutations have been obtained in \cite{kPerms}, however we obtain a more complete solution in this paper.  In Section \ref{Permutations} we establish a complete result on the existence of $s$-ocycles for permutations, and in Section \ref{Related} we explore similar results for functions and juggling sequences.

\section{Permutations}\label{Permutations}

We begin with a general lemma that will be used for both permutations and juggling sequences.  For a string $X = x_0x_1 \ldots x_{n-1}$, define the \textbf{rotation function} as follows.  A rotation by $s$ is given by: $$\rho^s(X) = x_sx_{s+1} \ldots x_{n-1}x_0x_1 \ldots x_{s-1}.$$  The following lemma shows that for a string of length $n$, rotations by $s$ partition the string into \textbf{blocks} of length gcd$(n,s)$, and we can always perform repeated $s$-rotations to start $X$ with any block desired.

\begin{lemma}\label{rotations}
    Let $n,s \in \mathbb{Z}^+$ with $1 \leq s \leq n-1$ and gcd$(n,s)=d$.  Consider a string $X=x_0x_1 \ldots x_{n-1}$, also written as $X = Y_0Y_1 \ldots Y_{m-1}$ where $n=md$ and $Y_i = x_{id}x_{id+1} \ldots x_{id+d-1}$.  Then for any $i \in \{0, 1, \ldots, m-1\}$ there is some $j \in \{0,1, \ldots , m-1\}$ such that: $$Y_iY_{i+1} \ldots Y_{m-1}Y_0Y_1 \ldots Y_{i-1} = \rho^{js}(X).$$
\end{lemma}
\begin{proof}
    Suppose that $s=kd$.  Then each rotation $\rho^s(X)$ advances through $k$ blocks of $X$.  Thus if we can show that gcd$(m,k)=1$, we are done.  However we note that gcd$(n,s)=d$ implies that there are integers $p,q$ such that $pn+qs=d$.  This implies the following.
    \begin{eqnarray*}
        pn+qs & = & d \\
        pmd + qkd & = & d \\
        pm + qk & = & 1
    \end{eqnarray*}
    Thus the same integers $p,q$ provide confirmation that gcd$(m,k)=1$.

\end{proof}

\begin{thm}\label{MainPerms}
    Let $n,s \in \mathbb{Z}^+$ with $1 \leq s \leq n-1$ and let $M$ be a multiset of size $n$.  There exists an $s$-ocycle on permutations of $M$ if and only if $n-s > \hbox{gcd}(n,s)$.
\end{thm}
\begin{proof}
    Define gcd$(n,s)=d$, and suppose that $n-s = kd > d$.  Construct the transition digraph $D$ with vertices representing $s$-permutations of $M$ and edges representing permutations of $M$.  We will show that this graph is balanced and connected, and hence eulerian.  Recall that an Euler tour in this graph corresponds to an $s$-ocycle for permutations of $M$.
    \begin{description}
        \item[Balanced:]  Let $X^{s-}=x_0x_1 \ldots x_{s-1}$ be an arbitrary vertex in the transition graph. For each possible out-edge corresponding to an $n$-permutation $X=x_0x_1 \ldots x_{n-1}$, we also have the in-edge $X'=x_{s}x_{s+1} \ldots x_{n-1}x_0x_1 \ldots x_{s-1}$.  Thus there is a one-to-one correspondence between in- and out-edges at each vertex, and so D is balanced.
        \item[Connected:]  Consider an arbitrary vertex $X^{s-}=x_0x_1 \ldots x_{s-1}$, which is an $s$-prefix of some permutation $X = x_0x_1 \ldots x_{n-1}$.  We will consider a partition of $X$ into blocks of size $d$, i.e. $X=X_0X_1 \ldots X_{m-1}$ where $X_i = x_{id}x_{id+1} \ldots x_{(i+1)d-1}$.  Our goal is to show that we can permute elements in any $k$ consecutive blocks $X_iX_{i+1} \ldots X_{i+k-1}$, which illustrates that any adjacent transpositions are possible, and hence we can reach all permutations.

            In D, rotations of $X$ correspond to the following cycle, where subscripts are computed modulo $n$.
            \begin{eqnarray*}
                x_0x_1 \ldots x_{s-1} & \rightarrow & x_{n-s}x_{n-s+1} \ldots x_{n-1} \\
                & \rightarrow & x_{n-2s}x_{n-2s+1} \ldots x_{n-s-1} \\
                & \rightarrow & x_{n-3s}x_{n-3s+1} \ldots x_{n-2s-1} \\
                & \vdots & \\
                & \rightarrow & x_{n-is}x_{n-is+1} \ldots x_{n-(i-1)s-1} \\
                & \vdots & \\
                & \rightarrow & x_0x_1 \ldots x_{s-1}
            \end{eqnarray*}
            Clearly from the vertex $X^{s-}$ we can permute all elements in $\{x_s, x_{s+1}, \ldots , x_{n-1}\}$ to determine an out-edge, and at any point we can permute the $(n-s)$-suffix of a permutation.  We would like to show that, through rotations, the $(n-s)$-suffix may contain any $k$ consecutive blocks from $\{X_0, X_1, \ldots , X_{m-1}\}$.  Consider the vertex $x_{n-is}x_{n-is+1} \ldots x_{n-(i-1)s-1}$.  How does the starting index $n-is$ correspond to $d$?
            \begin{eqnarray*}
                n-is & \equiv & in-is \pmod n \\
                & \equiv & i(n-s) \pmod n \\
                & \equiv & ikd \pmod n
            \end{eqnarray*}
            Thus the vertices in this rotation cycle always begin with a multiple of $kd$.  Our final step is to show that any multiple of $d$, say $id$, may be written as a multiple of $kd$ modulo $n$, which is done by Lemma \ref{rotations}.  Then we can rotate to start with any block desired, which is equivalent to pushing any $k$ consecutive blocks to the $(n-s)$-suffix.

            To summarize, we can perform adjacent transpositions $x_{i} \leftrightarrow x_{i+1}$ within $X^{s-}$ by rotating $X$ until $x_i$ and $x_{i+1}$ fall into the $k$ blocks in the $(n-s)$-suffix and then transposing.  Finally, by continuing along through rotations we will arrive at the vertex $x_0x_1 \ldots x_{i-1}x_{i+1}x_ix_{i+2} \ldots x_{s-1}$.  Thus adjacent transpositions are always connected, hence all permutations can be reached.
    \end{description}

    For the converse, suppose that $n-s = d = \hbox{gcd}(n,s)$.  In this case, rotations of the permutation $X=x_0x_1 \ldots x_{n-1}$ provide an $(n-s)$-suffix of length $d$ -- just one block from $X = X_0X_1 \ldots X_{m-1}$.  Thus we can always permute elements within each block, however the cyclic order of the blocks is fixed and we can perform the swap $x_i \leftrightarrow x_j$ if and only if $x_i$ and $x_j$ are in the same block.  Thus we are able to permute elements within blocks, but not permute blocks (only rotate the block order).  Hence the transition digraph connects only permutations with the same block order (rotations allowed).  Permutations with block order that are not simple rotations are not connected, so the graph is disconnected and no Euler tour exists.
\end{proof}

Note that Theorem \ref{MainPerms} agrees with the following well-known fact about universal cycles, or $(n-1)$-ocycles.

\begin{corr}
    There is no universal cycle for permutations of $M$.
\end{corr}
\begin{proof}
    In this case, we apply Theorem \ref{MainPerms} with $s=n-1$.  Then gcd$(n,s)=1=n-s$, so no $(n-1)$-ocycle exists.
\end{proof}

\section{Related Objects}\label{Related}

\subsection{Functions}

Many times injective, surjective, and bijective functions are represented by permutations.  We have the following facts about functions and theorems corresponding to their alternate representations.

\begin{fact}
    An injective function $f:[k] \rightarrow [n]$ may be represented by $x_1x_2 \ldots x_k$, the $k$-permutation of $[n]$ defined so that $x_i = f(i)$.
\end{fact}
\begin{thm}
    \emph{\cite{kPerms}}  For all $n,s,k \in \mathbb{Z}^+$ with $1 \leq s < k < n$, there is an $s$-ocycle for $k$-permutations of $[n]$.
\end{thm}

\begin{fact}
    A surjective function $f:[n] \rightarrow [h]$ may be represented by $x_1x_2 \ldots x_n$, the string with ground set $[h]$ defined so that $x_i = f(i)$.
\end{fact}

In \cite{ucwo}, it is shown that surjective functions are also represented by weak orders of $[n]$ with height $h-1$.  This observation is followed by the following theorem.

\begin{thm}
    \emph{\cite{ucwo}}  For all $n,s,h \in \mathbb{Z}^+$ with $1 \leq s \leq n-2$, gcd$(s,n)=1$, and $0 \leq h \leq n-1$, there is an $s$-ocycle for $\mathcal{W}(n,h)$.
\end{thm}

We are able to improve this theorem as follows.

\begin{thm}
    For all $n,s,h \in \mathbb{Z}^+$ with $1 \leq s \leq n-2$ and $h \leq n-1$ there is an $s$-ocycle for strings with ground set $[h]$.
\end{thm}
\begin{proof}
    We will show that the corresponding transition graph is eulerian.
    \begin{description}
        \item[Balanced:]  Consider a vertex $X^{s-} = x_1x_2 \ldots x_s$.  $X^{s-}$ is an $s$-prefix of the string $X=x_1x_2 \ldots x_n$ where $X$ has ground set $[h]$.  Since $x_{s+1}x_{s+2} \ldots x_nx_1x_2 \ldots x_s$ is also a string with ground set $[h]$, there is a bijection between in- and out-edges at $X^{s-}$.  Hence the graph is balanced.
        \item[Connected:]  Define the minimum vertex $V^{s-}$ to be the $s$-prefix of the permutation $V = 12 \cdots mm \cdots m$.  Let $X = x_1x_2 \ldots x_n$ be an arbitrary multiset permutation with ground set $[h]$, and let $X^{s-}$ be the $s$-prefix.  We will show a path from $X^{s-}$ to $V^{s-}$
            exists in the transition graph.

            Compare $X^{s-}$ and $V^{s-}$, and let $i$ be the least index such that $X^{s-}(i) \neq V^{s-}(i)$.  Note that since $h \leq n-1$, some element from $[h]$ must appear at least twice in $X$.  We will refer to any element appearing more than once as a \textbf{duplicate}.  We have two cases.

            \begin{enumerate}
                \item  If the letter $x_i \in [h]$ appears twice in $X$:

                Let $d = \hbox{gcd}(n,s)$, and rotate $X$ until the $d$-block containing $x_i$ is first.  If we rotate $X$ again by following an out-edge in the graph, we have arrived at a vertex $A$ representing an $s$-substring of $X$ without $x_i$.  Since $x_i$ also appears elsewhere in $X$, we can follow (backwards) the in-edge that is identical to $A$ except with $x_i$ replaced by $v_i$.  Now we are at an edge corresponding to an $s$-substring of $x_1x_2 \ldots x_{i-1}v_ix_{i+1} \ldots x_n$, so we are one step closer to the minimum vertex.

                \item  If the letter $x_i \in [h]$ appears exactly once in $X$:

                Since $x_i$ is not a duplicate in $X$, some other letter $x_j \in [h]$ is a duplicate.  In this case, we proceed as in case 1 to replace $x_j$ with the letter $x_i$.  Then $x_i$ is a duplicate so we may follow case 1 again to replace $x_i$ with $v_i$.  At this point we have moved closer to the minimum vertex.
            \end{enumerate}

            Continuing until we have transformed $X^{s-}$ to $V^{s-}$ produces a path from $X^{s-}$ to the minimum vertex.  Hence the graph is connected.
    \end{description}
    As the transition graph is balanced and connected, it is eulerian by Theorem \ref{euler}.
\end{proof}

\begin{fact}
    A bijective function $f:[n] \rightarrow [n]$ may be represented by $x_1x_2 \ldots x_n$, the permutation of $[n]$ defined so that $x_i = f(i)$.
\end{fact}

\begin{thm}
    Let $n,s \in \mathbb{Z}^+$ with $1 \leq s \leq n-1$.  There exists an $s$-ocycle on permutations of $[n]$ if and only if $n-s>\hbox{gcd}(n,s)$.
\end{thm}
\begin{proof}
    Use Theorem \ref{MainPerms} with $M = [n]$.
\end{proof}

\subsection{Juggling Sequences}\label{JS}

We begin with some lemmas that will help us to prove our main result.

\begin{lemma}\label{rotations}
    Let $R=r_0r_1 \ldots r_{n-1}$ be a string, and let $1 \leq s \leq n$.  Then $R' = \rho^s(R) = r_{0-s}r_{1-s} \ldots r_{n-1-s}$ (where addition is modulo $n$) is a valid juggling sequence if and only if $R$ is a valid juggling sequence.
\end{lemma}
\begin{proof}
    We will check the corresponding permutation sequence for $R'$ and show that it is valid.  Suppose for a contradiction that there are $i,j \in \{0, 1, \ldots , n-1\}$ with $$r'_{i}+i \equiv r'_{j}+j \pmod n.$$  Then we have:
    \begin{eqnarray*}
        r'_i+i & \equiv & r'_j + j \pmod n \\
        r_{i-s}+i & \equiv & r_{j-s}+j \pmod n \\
        r_{i-s}+i-s & \equiv & r_{j-s} + j-s \pmod n \\
        r_k + k & \equiv & r_\ell + \ell \pmod n
    \end{eqnarray*}
    Thus $\Pi_{R'}$ is valid if and only if $\Pi_R$ is valid.
\end{proof}

\begin{lemma}\label{perms}
    Let $T = t_0t_1 \ldots t_{n-1}$ be a juggling sequence, and let $0 \leq i,s \leq n-1$.  Then $$\Pi_{\rho^s(T)}(i) = \rho^s(\Pi_T)(i)-s \pmod n.$$
\end{lemma}
\begin{proof}
    \begin{eqnarray*}
        \Pi_{\rho^s(T)}(i) & = & \rho^s(T)(i) + i \pmod n \\
        & = & T(i+s)+i \pmod n \\
        & = & T(i+s) + i+s-s \pmod n \\
        & = & \Pi_T(i+s)-s \pmod n \\
        & = & \rho^s(\Pi_T)(i)-s \pmod n
    \end{eqnarray*}
\end{proof}

\begin{lemma}\label{reductions}
    Fix $n,s,b \in \mathbb{Z}^+$ with $1 \leq s \leq n-1$.  Define $D$ to be the $s$-ocycle transition digraph for juggling sequences of length $n$ using at most $b$ balls.  From any vertex $v_0v_1 \ldots v_{s-1}$, there exists a path to any vertex $v_0'v_1' \ldots v_{s-1}'$ where $v_i' \equiv v_i \pmod n$.
\end{lemma}
\begin{proof}
    Since adding/subtracting $n$ to any digit in a juggling sequence of length $n$ does not invalidate the sequence, if $v_i \geq n$ we can simply rotate some juggling sequence $X$ with $s$-prefix $v_0v_1 \ldots v_{s-1}$ until $v_i$ is in the $(n-s$)-suffix, replace with $v_i$ with $v_i - n$, and then rotate back to our original $s$-prefix with $v_i$ replaced by $v_i-n$.  Repeating this eventually will find a path to $v_1'v_2' \ldots v_s'$.
\end{proof}

\begin{thm}\label{JSMain}
    Fix $n,s,b \in \mathbb{Z}^+$ such that $1 \leq s \leq n-1$.  There exists an $s$-ocycle for the set of juggling sequences with period $n$ and at most $b$ balls if and only if $n-s > \hbox{gcd}(n,s)$.
\end{thm}
\begin{proof}
    We prove the forward direction by showing that the $s$-ocycle transition digraph $D$ has an Euler tour.  By Theorem \ref{euler} this is done by showing that the graph is balanced and connected.
    \begin{description}
        \item[Balanced:]  Consider a vertex $R=r_0r_1 \ldots r_{s-1}$ in $D$.  We will show that any $(n-s)$-string $Q=q_0q_1 \ldots q_{n-s-1}$ that is a valid $s$-suffix for $R$ is also a valid $s$-prefix for $R$.  Note that strings $RQ$ and $QR$ are simply rotations of each other, so by Lemma \ref{rotations} either both strings are valid juggling sequences or neither string is.  In this manner, there is a bijection between in- and out-edges, hence all vertices are balanced.

        \item[Connected:]  Consider an arbitrary vertex $T^{s-}=t_0t_1 \ldots t_{s-1}$ that is an $s$-prefix to some juggling sequence $T=t_0t_1 \ldots t_{n-1}$.  First, by Lemma \ref{reductions} we may assume that $t_i \in \{0,1, \ldots , n-1\}$ for all $i \in \{0, 1, \ldots , n-1\}$.  We will show that this arbitrary vertex is connected to the \textbf{minimum vertex}, which we define to be $V^{s-}=0^s$ (a prefix of $V = 0^n$).  In doing so, we will have shown that every vertex is connected to $V^{s-}$ and hence the graph is connected.

            Compare permutation sequences $\Pi_T$ and $\Pi_V$ corresponding to juggling sequences $T$ and $V$, respectively.  Note that $\Pi_V = 012\ldots (n-1)$, so suppose that for all $x \in \{0, 1, \ldots , i-1\}$ we have $\Pi_T(x) = \Pi_V(x) = x$, but that $\Pi_T(i) = j$ for some $j \in \{i+1, i+2, \ldots , s-1\}$.  We will find a path from the $s$-prefix of $T$ to the $s$-prefix of some juggling sequence $T'$ with permutation sequence $\Pi_{T'}$ that agrees with $\Pi_V$ in the first $i+1$ positions.  Repeating this procedure until $i=n$ will construct a path through $D$ from $T^{s-}$ to $V^{s-}$.

            Assume $n-s=kd$ and $n=md$ for integers $m,k$ where $d= \hbox{gcd}(n,s)$, and let $T=Y_0Y_1 \ldots Y_{m-1}$ be a partition of $T$ into $d$-blocks, i.e. $Y_a=t_{ad}t_{ad+1} \ldots t_{ad+d-1}$.  Suppose that $t_i \in Y_a$ and $t_j \in Y_b$.  We have two cases.
            \begin{enumerate}
                \item  If $b-a < k$:

                In this case, from Lemma \ref{rotations} we may perform $s$-rotations on $T$ so that we arrive at a vertex with both $t_i$ and $t_j$ in the $(n-s)$-suffix of $T$.  Suppose that we performed a total rotation of size $R$, i.e. we are now at the vertex that represents the $s$-prefix of $\rho^R(T)$.

                At this point, the values $t_i$ and $t_j$ are located in positions $i-R$ and $j-R$ of $\rho^R(T)$, respectively.  By Lemma \ref{perms} the corresponding permutation sequence entries are:
                \begin{eqnarray*}
                    \Pi_{\rho^R(T)}(i-R) & = & \rho^R(\Pi_T)(i-R)-R \pmod n \\
                    & = & \Pi_T(i)-R \pmod n
                \end{eqnarray*}
                and
                \begin{eqnarray*}
                    \Pi_{\rho^R(T)}(j-R) & = & \Pi_T(j)-R \pmod n.
                \end{eqnarray*}

                 We now note that the vertex $(\rho^R(T))^{s-}$, defined as the $s$-prefix of $\rho^R(T)$, is also the $s$-prefix of the juggling sequence $\rho^R(T')$ that is obtained from $\rho^R(T)$ by performing the swap $\Pi_{\rho^R(T)}(i-R) \leftrightarrow \Pi_{\rho^R(T)}(j-R)$ and adjusting the values $\rho^R(T)(i-R)$ and $\rho^R(T)(j-R)$ appropriately.  Then rotating backwards by $R$ we reach a juggling sequence $T'$ with $\Pi_{T'}(x)=x$ for all $x \in \{0,1, \ldots , i-1, i\}$.  This means that we have found a path from $T^{s-}$ to $T'^{s-}$, which is one step closer to the minimum vertex.

                \item  If $b-a \geq k$:

                In this case the blocks containing $t_i$ and $t_j$ are more than $k$ apart, so we cannot rotate to have both $t_i$ and $t_j$ in the $(n-s)$-suffix of $T$ simultaneously.  Instead, we pick a point $t_z$ in the block $Y_c$ that is $k$ blocks preceding $Y_b$ and follow case (1) with $t_z$ in place of $t_i$.  This produces a juggling sequence $T'$ with $\Pi_{T'}(z)=i$, where $i \leq z < j$.  Repeating, we will eventually have moved the permutation sequence value $i$ to a block that is close enough to block $Y_a$ to apply Case (1).
            \end{enumerate}
            By repeating the above procedures, we will eventually have transitioned to the vertex corresponding to a juggling sequence with $s$-prefix $01 \ldots (s-1)$.  At this point we have reached the minimum vertex and we are done.
    \end{description}

    For the reverse direction, suppose that $n-s = \hbox{gcd}(n,s)=d$.  Recall that the number of balls $b \in \mathbb{Z}^+$ is given by: $$b = \frac{1}{n} \sum_{i=0}^{n-1}t_i.$$  Thus juggling sequences of period $n$ must always satisfy $\sum_{i=0}^{n-1}t_i \equiv 0 \pmod n$.  Equivalently, when we partition a juggling sequence $T=t_0t_1 \ldots t_{n-1}$ into blocks of length $d$, i.e. $T=Y_0Y_1 \ldots Y_{m-1}$, where $n=md$, we must have: $$\sum_{i=0}^{m-1} w(Y_i) = \sum_{i=0}^{m-1} \sum_{j=0}^d t_{id+j} = \sum_{i=0}^{n-1}t_i \equiv 0 \pmod n,$$ where, for a given block $Y_i$, we call $w(Y_i)=\sum_{j=0}^{d}t_{id+j}$ the \textbf{weight} of block $Y_i$.

    Now since $n-s = \hbox{gcd}(n,s)$, from vertex $T^{s-} = t_0t_1 \ldots t_{s-1}$ in $D$ we may only move to vertices in which the $(n-s)$-suffix has weight equivalent to $w(Y_{m-1})$ mod $n$.  Thus if we can show that, for any $n,s,$ and $b$, there exists a juggling sequence with a block with weight $w \not \equiv 0 \pmod n$, then we are done.  This is witnessed by the juggling sequence $$T = \begin{array}{ccccccccc} d & 0 & 0&  \ldots & 0 & (n-d) & 0 & \ldots & 0,\end{array}$$ with permutation sequence $$\Pi_T = \begin{array}{cccccccccc} d & 1 & 2 & \ldots & (d-1) & 0 & (d+1) & (d+2) & \ldots & (n-1). \end{array}$$  This juggling sequence utilizes one ball (recall we required $b \in \mathbb{Z}^+$) and the weight of the first block of length $d$ is $d \not \equiv 0 \pmod n$. Thus this first block must always have weight equal to $d$ modulo $n$ if $n-s = d$, and so the vertex $T^{s-}$ that represents the $s$-prefix of $T$ is not connected to the vertex $0^s$.  Hence no Euler tour can exist and so no $s$-ocycle exists.
\end{proof}

While Theorem \ref{JSMain} completes the question of when $s$-ocycles for juggling sequences of period $n$ and at most $b$ balls exist, several variations remain open.  For example, one might consider juggling sequences with:
\begin{itemize}
    \item exactly $b$ balls,
    \item at least $b$ balls,
    \item fixed minimum period, or
    \item period $n$ with no restriction on number of balls.
\end{itemize}

\end{document}